\documentclass[12pt]{article}
\usepackage{graphics}
\usepackage{amsmath}
\usepackage{amsfonts}
\usepackage{enumerate}
\usepackage{latexsym}
\usepackage{epsfig}

\newtheorem{theorem}{Theorem}[section]
\newtheorem{lemma}[theorem]{Lemma}

\def\proofbox{\begin{picture}(6.5,6.5)
\put(0,0){\framebox(6.5,6.5){}}\end{picture}}
\newenvironment{proof}{\noindent{\it Proof.\quad}}{\hfill\proofbox}

\begin{document}

\title{Superinjective Simplicial Maps of the Complexes of Curves on Nonorientable Surfaces}
\author{Elmas Irmak}

\maketitle

\renewcommand{\sectionmark}[1]{\markright{\thesection. #1}}


\thispagestyle{empty}
\maketitle



\begin{abstract} We prove that each superinjective simplicial map of the complex of curves
of a compact, connected, nonorientable surface is induced
by a homeomorphism of the surface, if $(g, n) \in \{(1, 0), (1, 1), (2, 0), (2, 1), (3, 0)\}$ or $g + n \geq 5$,
where $g$ is the genus of the surface and $n$ is the number of the boundary components. \end{abstract}

\maketitle

{\small Key words: Mapping class groups, simplicial maps, nonorientable surfaces

}

\section{Introduction}

Let $N$ be a compact, connected, nonorientable surface of genus $g$ (connected sum of $g$ copies of projective planes)
with $n$ boundary components. Mapping class group, $Mod_N$, of $N$ is defined to be the group of isotopy classes of all self-homeomorphisms of $N$.
The \textit{complex of curves}, $\mathcal{C}(N)$, on $N$ is an abstract simplicial complex defined as follows: A simple closed curve
on $N$ is called {\it nontrivial} if it does not bound a disk, a mobius band, and it is not isotopic to a boundary component of $N$.
The vertex set, $A$, of $\mathcal{C}(N)$ is the set of isotopy classes of nontrivial simple closed curves on $N$. A set of vertices forms
a simplex in $\mathcal{C}(N)$ if they can be represented by pairwise disjoint simple closed curves. The geometric intersection number
$i([a], [b])$ of $[a]$, $[b] \in A$ is the minimum number of points of $x \cap y$ where $x \in [a]$ and $y \in [b]$.

A simplicial map $\lambda : \mathcal{C}(N) \rightarrow \mathcal{C}(N)$ is called {\it superinjective}
if it satisfies the following condition: If $[a], [b]$ are two vertices in $\mathcal{C}(N)$, then
$i([a], [b]) = 0$ if and only if $i(\lambda([a]), \lambda([b])) = 0$. Superinjective simplicial maps
preserve geometric intersection zero and nonzero properties by the definition. Superinjective simplicial maps send the
isotopy class of a 1-sided simple closed curve to the isotopy class of a 1-sided simple closed curve, and send the isotopy class of a
2-sided simple closed curve to the isotopy class of a 2-sided simple closed curve on $N$ (see Lemma \ref{easy}). Superinjective
simplicial maps are injective. This follows from the result of Theorem \ref{A} for some small genus cases. The proof is given in Lemma \ref{inj} for the remaining cases.\\

\protect\nopagebreak\noindent\rule{1.5in}{.01in}{\vspace{0.005in}}

\noindent {\small The author was supported by Faculty Research Incentive Grant, BGSU.}\\

The main result of this paper is the following:

\begin{theorem} Let $N$ be a compact, connected, nonorientable surface of genus $g$ with
$n$ boundary components. Suppose that either $(g, n) \in \{(1, 0), (1, 1), (2, 0), (2, 1), (3, 0)\}$ or $g + n \geq 5$.
If $\lambda : \mathcal{C}(N) \rightarrow \mathcal{C}(N)$ is a superinjective simplicial map, then $\lambda$ is
induced by a homeomorphism $h : N \rightarrow N$ (i.e $\lambda([a]) = [h(a)]$ for every vertex $[a]$ in $\mathcal{C}(N)$).\end{theorem}

We note that $g + n = 4$ case is open where $g$ is the genus of the surface and $n$ is the number of the boundary components.

The mapping class group and the complex of curves on orientable surfaces are defined similarly as follows: Let $R$ be a compact,
connected, orientable surface. Mapping class group, $Mod_R$, of $R$ is defined to be the group of isotopy classes of orientation
preserving homeomorphisms of $R$. Extended mapping class group, $Mod_R^*$, of $R$ is defined to be the group of isotopy classes
of all self-homeomorphisms of $R$. The complex of curves $\mathcal{C}(R)$, on $R$ was introduced by Harvey in \cite{Har}. It is defined
as an abstract simplicial complex. The vertex set is the set of isotopy classes of nontrivial simple closed curves, where nontrivial
means it does not bound a disk and it is not isotopic to a boundary component of $R$. A set of vertices forms a simplex
in $\mathcal{C}(R)$ if they can be represented by pairwise disjoint simple closed curves.

Ivanov proved that the automorphism group of the curve complex is isomorphic to the extended mapping class group on orientable surfaces.
As an application he proved that isomorphisms between any two finite index subgroups are geometric, see \cite{Iv1}. Ivanov's results were proven by
Korkmaz in \cite{K1} for lower genus cases. Luo gave a different proof of these results for all cases in \cite{L}. Ivanov and McCarthy classified injective homomorphisms between the mapping class groups of compact, connected, orientable surfaces of genus $\geq 2$ in \cite{IMc}.
After Ivanov's work about automorphisms of complexes of curves, mapping class group was viewed as the automorphism group
of various geometric objects on orientable surfaces. These objects include Schaller's complex (see \cite{S} by Schaller), the
complex of pants decompositions (see \cite{M} by Margalit), the complex of nonseparating curves (see \cite{Ir3} by Irmak), the complex
of separating curves (see \cite{BM1} by Brendle-Margalit, and \cite{MV} by McCarthy-Vautaw), the complex of Torelli geometry (see \cite{FIv}
by Farb-Ivanov), the Hatcher-Thurston complex (see \cite{IrK} by Irmak-Korkmaz), and the complex of arcs (see \cite{IrM} by Irmak-McCarthy).
As applications, Farb-Ivanov proved that the automorphism group of the Torelli subgroup is isomorphic to the mapping class group in \cite{FIv},
and McCarthy-Vautaw extended this result to $g \geq 3$ (where $g$ is the genus of the surface) in \cite{MV}.

On orientable surfaces: Irmak proved that superinjective simplicial maps of the curve complex are induced by homeomorphisms of the surface
to classify injective homomorphisms from finite index subgroups of the mapping class group to the whole group (they are geometric except
for closed genus two surface) for genus at least two in \cite{Ir1}, \cite{Ir2}, \cite{Ir3}. Behrstock-Margalit and Bell-Margalit proved
these results for lower genus cases in \cite{BhM} and in \cite{BeM}. Brendle-Margalit proved that superinjective simplicial maps of
separating curve complex are induced by homeomorphisms, and using this they proved that an injection from a finite index subgroup of
$K$ to the Torelli group, where $K$ is the subgroup of mapping class group generated by Dehn twists about separating curves, is induced
by a homeomorphism in \cite{BM1}. Shackleton proved that injective simplicial maps of the curve complex are induced by
homeomorphisms in \cite{Sh} (he also considers maps between different surfaces), and he obtained strong local co-Hopfian results
for mapping class groups.

On nonorientable surfaces: For odd genus cases, Atalan proved that the automorphism group of the curve complex is isomorphic
to the mapping class group if $g + r \geq 6$ (where $g$ is the genus of the surface and $r$ is the number of boundary components)
in \cite{A}. Irmak proved that each injective simplicial map from the arc complex
of a compact, connected, nonorientable surface with nonempty boundary to itself is induced by a homeomorphism of the surface
in \cite{Ir4}. She also proved that the automorphism group of the arc complex is isomorphic to the quotient of the mapping class
group of the surface by its center. Atalan-Korkmaz proved that the automorphism group of the curve complex is isomorphic to the
mapping class group in \cite{AK} if $g + r \geq 5$ where $g$ is the genus of the surface and $r$ is the number of boundary components.
They also proved that two curve complexes are isomorphic if and only if the
underlying surfaces are homeomorphic. In this paper we use some results from \cite{AK}.


\section{Some small genus cases}

In this section we will prove our main results for $(g, n) \in \{(1, 0), (1, 1), (2, 0), (2, 1)\}$.
We note that
a simple closed curve is called {\it 1-sided} if its regular neighborhood is a Mobius band. It is called {\it 2-sided}
if its regular neighborhood is an annulus.
There are 1-sided simple closed curves as well as 2-sided simple closed curves on nonorientable surfaces.
Some 1-sided simple closed curves have orientable complements and some have nonorientable complements. We first prove the following
lemma for any compact, connected, nonorientable surface.

\begin{figure}
\begin{center}
\hspace{0.006cm} \epsfxsize=1.1in \epsfbox{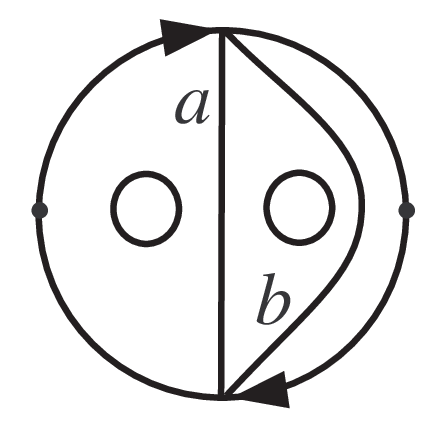} \hspace{1.25cm} \epsfxsize=2.37in \epsfbox{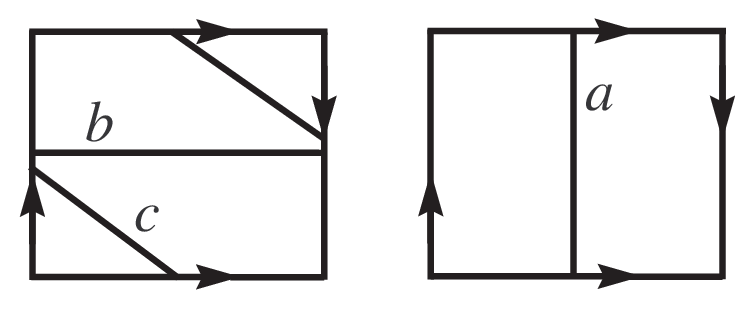}

\small{(i) \hspace{5.23cm} (ii) \hspace{1.1cm}}
\caption{Vertices of $\mathcal{C}(N)$ for (i) $(g, n)=(1,2)$, and (ii) $(g, n)=(2,0)$}
\label{figure1}
\end{center}
\end{figure}

\begin{lemma}
\label{easy} Let $N$ be a compact, connected, nonorientable surface. Let $\lambda : \mathcal{C}(N) \rightarrow \mathcal{C}(N)$
be a superinjective simplicial map. If $a$ is a 1-sided simple closed curve on $N$, then $\lambda([a])$ is the isotopy class of a 1-sided simple closed curve on $N$. If $a$ is a 2-sided simple closed curve on $N$, then $\lambda([a])$ is the isotopy class of a 2-sided simple closed curve on $N$.
\end{lemma}

\begin{proof} A simple closed curve $a$ is 1-sided if and only if $i([a], [a]) =1$. It is 2-sided if and only if $i([a], [a]) = 0$.
If $a$ is a 1-sided simple closed curve on $N$, we have $i([a], [a]) = 1$. Then $i(\lambda([a]), \lambda([a])) \neq 0$
since $\lambda$ is superinjective. So, $\lambda([a])$ is the isotopy class of a 1-sided simple closed curve on $N$. If $a$ is a 2-sided simple closed curve on $N$, we have $i([a], [a]) = 0$. Then $i(\lambda([a]), \lambda([a])) = 0$ since $\lambda$ is superinjective. This implies that $\lambda([a])$ is the isotopy class of a 2-sided simple closed curve on $N$.\end{proof}\\

If $(g, n) = (1, 2)$, there are only two vertices in the curve complex (see \cite{Sc}). They are the isotopy classes of
$a$ and $b$ as shown in Figure \ref{figure1}. The curves $a$ and $b$ are both 1-sided. We see that $i([a], [b]) = 1$. The simplicial map
sending both of these vertices to $[a]$ is superinjective and it is not induced by a homeomorphism.
In this case we can see that every injective simplicial map is induced by a homeomorphism. If the map fixes each of the vertices,
then it is induced by the identity homeomorphism, if it switches them then it is induced by a homeomorphism that switches the curves $a$ and $b$.

\begin{figure}
\begin{center}
\epsfxsize=1.2in \epsfbox{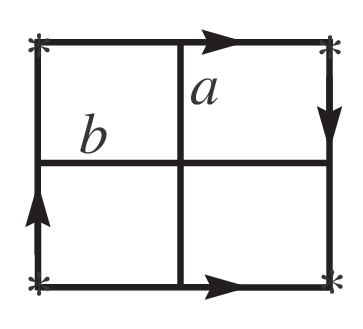} \hspace{0.45cm}
\epsfxsize=3.5in \epsfbox{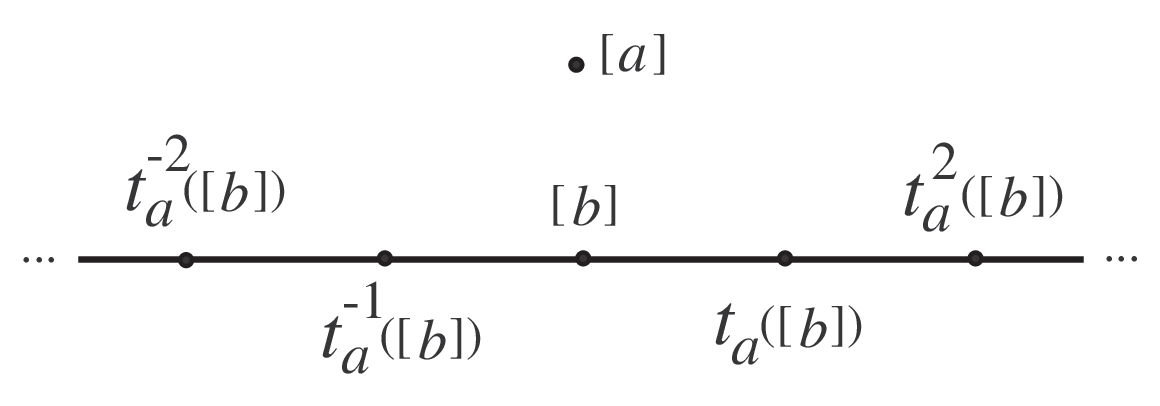}
\caption{Curve complex for $(g, n)=(2,1)$}
\label{figure2}
\end{center}
\end{figure}

\begin{theorem}
\label{A} Let $N$ be a compact, connected, nonorientable surface of genus $g$ with
$n$ boundary components. Suppose that $(g, n) \in \{(1, 0), (1, 1), (2, 0), (2, 1)\}$.
If $\lambda : \mathcal{C}(N) \rightarrow \mathcal{C}(N)$ is a superinjective simplicial map, then $\lambda$ is
induced by a homeomorphism $h : N \rightarrow N$ (i.e $\lambda([a]) = [h(a)]$ for every vertex $[a]$ in $\mathcal{C}(N)$).
\end{theorem}

\begin{proof}
If $(g, n) = (1, 0)$, $N$ is the projective plane. There is only one element (isotopy class of a 1-sided curve)
in the curve complex. Hence, any superinjective simplicial map is induced by the identity homeomorphism.
If $(g, n) = (1, 1)$, $N$ is Mobius band. There is only one element (isotopy class of a 1-sided curve) in the curve complex.
Hence, any superinjective simplicial map is induced by the identity homeomorphism.

If $(g, n) = (2, 0)$, there are only three vertices and one edge in the curve complex (see \cite{Sc}). The vertices are the isotopy
classes of $a$, $b$ and $c$ as shown in Figure \ref{figure1} (ii). The curve $a$ is 2-sided,
$b$ is 1-sided and $c$ is 1-sided. Since $a$ is the only 2-sided curve
among these three curves, $\lambda$ fixes $[a]$ and sends $\{[b], [c]\}$ to $\{[b], [c]\}$ by Lemma \ref{easy}. We see that $i([b], [c]) = 0$.
So, $i(\lambda([b]), \lambda([c]) = 0$. This implies that $\lambda([b]) \neq \lambda([c])$ as both are the isotopy classes of 1-sided curves.
If $\lambda$ fixes each of $[b]$ and $[c]$ then it is induced by the identity homeomorphism, if it switches $[b]$ and $[c]$ then
it is induced by a homeomorphism that switches the 1-sided curves $b$ and $c$, while fixing $a$ up to isotopy.

If $(g, n) = (2, 1)$, then the curve complex is given by Scharlemann in \cite{Sc} as follows: Let $a$ and $b$ be as in
Figure \ref{figure2}. We see that $i([a], [b]) = 1$. The vertex set of the curve complex is $\{[a], t_a^m ([b]) : m \in \mathbb{Z} \}$,
where $t_a$ is the Dehn twist about the 2-sided curve $a$. The complex is shown in Figure \ref{figure2}.
We see that $a$ is 2-sided, and each element in $\{t_a^m ([b]) : m \in \mathbb{Z} \}$ has a 1-sided representative. So,
$\lambda$ fixes $[a]$ and sends $\{t_a^m ([b]) : m \in \mathbb{Z} \}$ to itself by Lemma \ref{easy}.
For any $m$, the elements $t_a^m ([b])$ and $t_a^{m+1} ([b])$ cannot be send to the same element by $\lambda$ since they have different intersection information with $t_a^{m-1} ([b])$, and $\lambda$ preserves geometric intersection zero and nonzero by definition. Similarly, $t_a^m ([b])$ and $t_a^{n} ([b])$ cannot be send to the same element by $\lambda$ for any $m, n$ with $m \neq n$. Vertices connected by an edge go to vertices connected by an edge. So, either $\lambda(t_a^m ([b])) = t_a^k ([b])$,
$\lambda(t_a^{m-1} ([b]))= t_a^{k-1} ([b])$, $\lambda(t_a^{m+1} ([b]))= t_a^{k+1} ([b])$ or $\lambda(t_a^m ([b])) = t_a^k ([b])$,
$\lambda(t_a^{m-1} ([b]))= t_a^{k+1} ([b])$, $\lambda(t_a^{m+1} ([b]))= t_a^{k-1} ([b])$ for some $k \in \mathbb{Z}$. By cutting $N$ along $a$ we get a cylinder with one puncture as shown in Figure \ref{figure2a}. There is a reflection of the cylinder
interchanging front face with the back face which fixes $b$ pointwise. This gives us a homeomorphism $r$ of $N$ such that $r(b) = b$ and $r(a) = a^{-1}$. So, $r_\#([b]) = [b]$ and $r_\#(t_a([b])) = {t_a}^{-1}([b])$. The map $r_\#$ reflects the graph in Figure \ref{figure2} at $[b]$ and fixes $[a]$. Now it is easy to see that $\lambda$ is induced by $t_a^k$ or $t_a^k \circ r$ for some $k \in \mathbb{Z}$.
\end{proof}

\begin{figure}
\begin{center}
\epsfxsize=1.6in \epsfbox{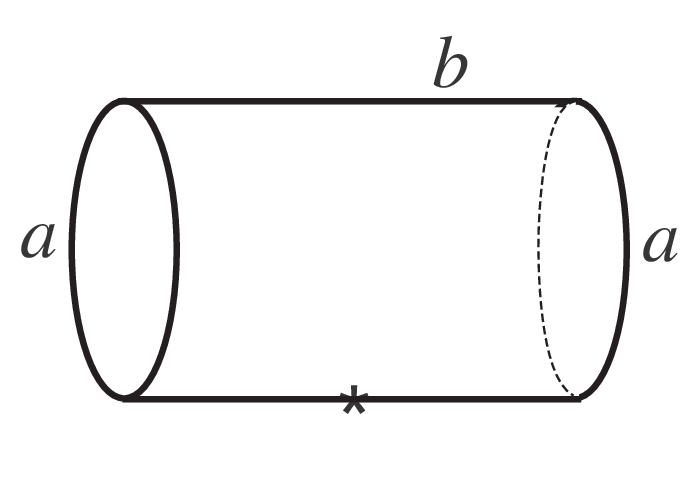}
\caption{Cutting along $a$}
\label{figure2a}
\end{center}
\end{figure}

\section{Properties of Superinjective Simplicial Maps}


Let $N$ be a compact, connected, nonorientable surface of genus $g$ with $n$ boundary components. We will list some properties of superinjective simplicial maps. First we give some definitions.

Let $P$ be a set of pairwise disjoint, nonisotopic, nontrivial simple closed curves on $N$. $P$ is called a {\it pair of pants decomposition}
of $N$ if each component $\Delta$ of the surface $N_P$, obtained by cutting $N$ along $P$, is a pair of pants. A pair of pants of a
pants decomposition is the image of one of these connected components under the quotient map $q: N_P \rightarrow N$.
Let $a$ and $b$ be two distinct elements in a pair of pants decomposition $P$. Then $a$ is called {\it adjacent} to $b$ w.r.t. $P$ iff
there exists a pair of pants in $P$ which has $a$ and $b$ on its boundary.

Let $P$ be a pair of pants decomposition of $N$. Let $[P]$ be the set of isotopy classes of elements of $P$.
Note that $[P]$ is a maximal simplex of $\mathcal{C}(N)$. Every maximal simplex $\sigma$ of $\mathcal{C}(N)$
is equal to $[P]$ for some pair of pants decomposition $P$ of $N$.

On orientable surfaces all maximal simplices have the same dimension $3g + n -4$ where $g$ is the genus and
$n$ is the number of boundary components of the orientable surface. On nonorientable surfaces this is not
the case. There are different dimensional maximal simplices. In Figure \ref{Fig0} we see some of them in
$g=7$ case. In the figure we see cross signs. This means that the interiors of the disks with cross signs
inside are removed and the antipodal points on the resulting boundary components are identified.

The following lemma is given in \cite{A} and \cite{AK}:

\begin{figure}
\begin{center}
\epsfxsize=2in \epsfbox{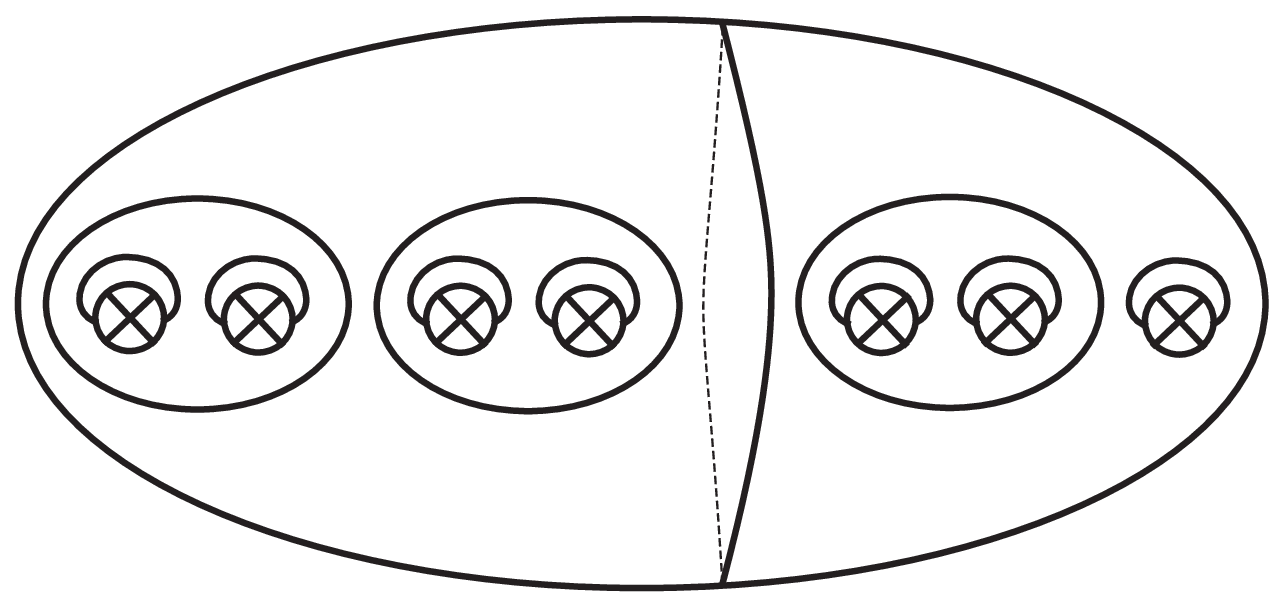} \hspace{0.1in} \epsfxsize=2in
\epsfbox{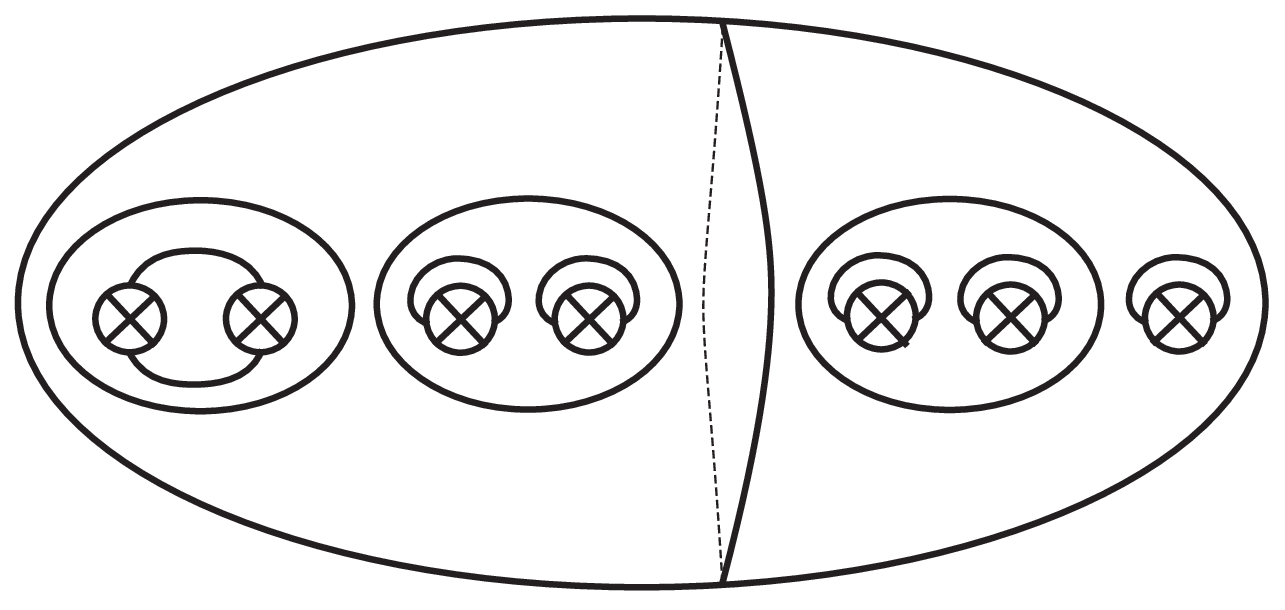} \hspace{0.1in} \epsfxsize=2in \epsfbox{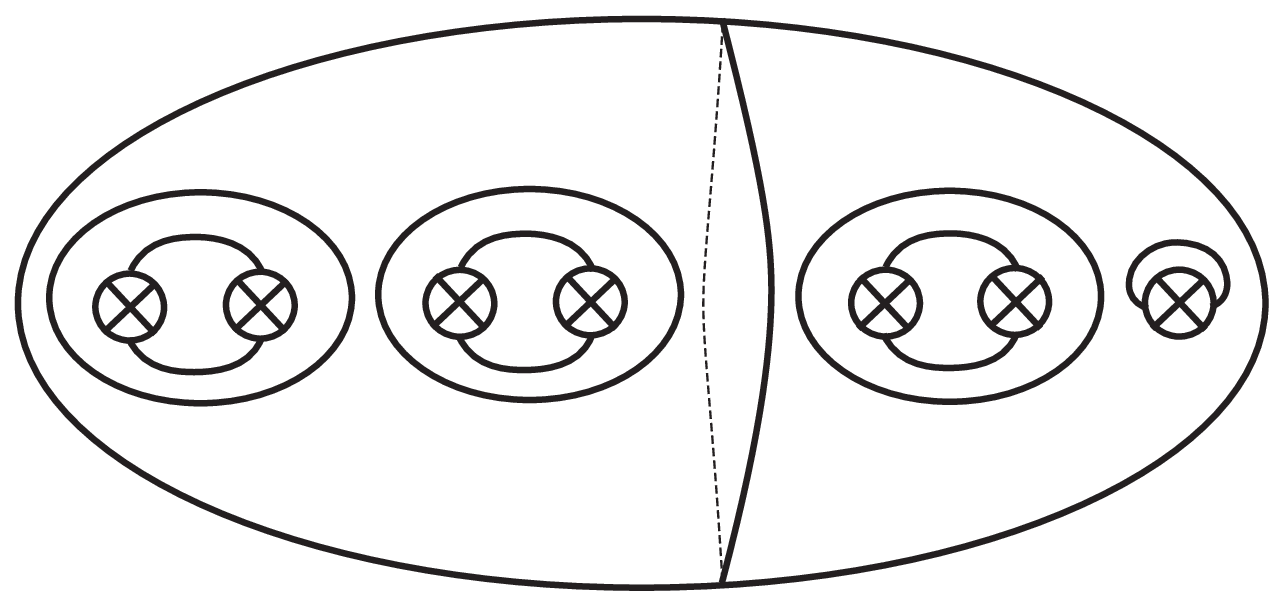}
\caption{Some maximal simplices on genus 7}
\label{Fig0}
\end{center}
\end{figure}

\begin{lemma}
\label{dim} Let $N$ be a nonorientable surface of genus $g \geq 2$ with $n$ boundary components. Suppose that $(g, n) \neq (2, 0)$.
Let $a_r= 3r+n-2$ and $b_r = 4r +n -2$ if $g = 2r+1$, and let $a_r= 3r+n-4$ and $b_r= 4r +n -4$ if $g =2r$.
Then there is a maximal simplex of dimension $q$ in $\mathcal{C}(N)$ if and only if $a_r \leq q \leq b_r$.\end{lemma}

Now we prove some properties of superinjective simplicial maps.

\begin{lemma}
\label{inj} Suppose that $(g, n)= (3, 0)$ or $g + n \geq 4$. A superinjective simplicial map
$\lambda : \mathcal{C}(N) \rightarrow \mathcal{C}(N)$ is injective.
\end{lemma}

\begin{proof} Let $[a]$ and $[b]$ be two distinct vertices in $\mathcal{C}(N)$. If one of them is 1-sided and the other is 2-sided, then
$\lambda$ cannot send them to the same element by Lemma \ref{easy}. Suppose $a$ and $b$ are both 2-sided. Then both $\lambda([a])$ and $\lambda([b])$ are the isotopy classes of 2-sided curves by Lemma \ref{easy}. If $i([a], [b]) \neq 0$, then
$i(\lambda([a]), \lambda([b])) \neq 0$, since $\lambda$ is superinjective. This implies that $\lambda([a]) \neq \lambda([b])$ as both are isotopy classes of 2-sided curves. If $i([a], [b]) = 0$, we choose a vertex $[c]$ of $\mathcal{C}(N)$ such that  $i([a], [c]) = 0$, and  $i([b], [c]) \neq 0$. Then $i(\lambda([a]), \lambda([c])) = 0$, $i(\lambda([b]), \lambda([c])) \neq 0$. So, $\lambda([a]) \neq \lambda([b])$.
Now suppose $a$ and $b$ are both 1-sided. Then both $\lambda([a])$ and $\lambda([b])$ are the isotopy classes of 1-sided curves by Lemma \ref{easy}.
We choose a vertex $[c]$ of $\mathcal{C}(N)$ such that  $i([a], [c]) = 0$, and  $i([b], [c]) \neq 0$ (we note that this choice of $c$ is not possible only in the case when $(g, n)=(1, 2)$, but we do not have that case in our statement). Then $i(\lambda([a]), \lambda([c])) = 0$, $i(\lambda([b]), \lambda([c])) \neq 0$. So, $\lambda([a]) \neq \lambda([b])$. Hence, $\lambda$ is injective.\end{proof}\\

If a maximal simplex has dimension $b_r = 4r +n -2$ if $g = 2r+1$, or $b_r = 4r +n -4$ if $g=2r$, we will call
it a {\it top dimensional maximal simplex}. Since a superinjective map is injective when $(g, n)= (3, 0)$ or $g + n \geq 4$,
it sends top dimensional maximal simplices to top dimensional maximal simplices.
In the following lemmas we will see that adjacency and nonadjacency are preserved w.r.t. top dimensional maximal simplices.

\begin{lemma}
\label{adjacent} Suppose that $(g, n)= (3, 0)$ or $g + n \geq 4$.
Let $\lambda : \mathcal{C}(N) \rightarrow \mathcal{C}(N)$ be a superinjective simplicial map.
Let $P$ be a pair of pants decomposition on $N$ which corresponds to a top dimensional maximal
simplex in $\mathcal{C}(N)$. Let $a, b \in P$ such that $a$ is adjacent to $b$ w.r.t.
$P$. There exists $a'  \in \lambda([a])$ and $b'  \in \lambda([b])$ such that $a'$ is adjacent
to $b'$ w.r.t. $P'$ where $P'$ is a set of pairwise disjoint curves representing $\lambda([P])$ containing $a', b'$.
\end{lemma}

\begin{proof} Suppose that $(g, n)= (3, 0)$ or $g + n \geq 4$.
Let $P$ be a pair of pants decomposition on $N$ which corresponds to a top
dimensional maximal simplex in $\mathcal{C}(N)$. The statement is easy to see in $(g, n)= (3, 0)$ and $(g, n) = (1,3)$ cases as
there are only three curves in $P$ if $(g, n)= (3, 0)$, and
there are only two curves in $P$ if $(g, n) = (1,3)$.

Assume that $(g, n) \neq (3, 0)$ and $(g, n) \neq (1,3)$. Let $a, b \in P$ such that $a$ is adjacent
to $b$ w.r.t. $P$. We can find a simple closed curve $c$ on $N$ such
that $c$ intersects only $a$ and $b$ nontrivially (with nonzero geometric intersection) and
$c$ is disjoint from all the other curves in $P$.
Let $P'$ be a set of pairwise disjoint curves representing $\lambda([P])$. Since $\lambda$ is
injective by Lemma \ref{inj}, $\lambda$ sends top dimensional maximal simplices to top
dimensional maximal simplices. So, $P'$ corresponds to a top dimensional maximal simplex.
Assume that $\lambda([a])$ and $\lambda([b])$ do not have adjacent representatives w.r.t. $P'$.
Since $i([c], [a]) \neq 0$ and $i([c], [b]) \neq 0$, we have $i(\lambda([c]), \lambda([a])) \neq 0$ and
$i(\lambda([c]), \lambda([b])) \neq 0$ by superinjectivity. Since  $i([c], [e]) = 0$ for all
$e \in P \setminus \{a, b\}$, we have $i(\lambda([c]), \lambda([e])) = 0$ for all
$e \in P \setminus \{a, b\}$. But this is not possible because $\lambda([c])$ has to intersect
geometrically essentially with some isotopy class other than $\lambda([a])$ and $\lambda([b])$ in $\lambda([P])$
to be able to make essential intersections with $\lambda([a])$ and $\lambda([b])$ since
$\lambda([P])$ is a top dimensional maximal simplex. This gives a contradiction to the assumption
that $\lambda([a])$ and $\lambda([b])$ do not have adjacent representatives.
\end{proof}

\begin{lemma}
\label{nonadjacent} Suppose that $g + n \geq 4$. Let $\lambda : \mathcal{C}(N) \rightarrow \mathcal{C}(N)$ be
a superinjective simplicial map. Let $P$ be a pair of pants decomposition on $N$ which corresponds to a top
dimensional maximal simplex in $\mathcal{C}(N)$. Let $a, b \in P$ such that $a$ is not adjacent to $b$
w.r.t. $P$. There exists $a'  \in \lambda([a])$ and $b'  \in \lambda([b])$ such that $a'$ is not adjacent
to $b'$ w.r.t. $P'$ where $P'$ is a set of pairwise disjoint curves representing $\lambda([P])$ containing $a', b'$.
\end{lemma}

\begin{proof} Suppose that $g + n \geq 4$. Let $P$ be a pair of pants decomposition on $N$ which corresponds
to a top dimensional maximal simplex in $\mathcal{C}(N)$. Let $a, b \in P$ such that $a$ is not adjacent to $b$
w.r.t. $P$. Let $P'$ be a set of pairwise disjoint curves representing $\lambda([P])$.
We can find distinct simple closed curves $c$ and $d$ on $N$ such that $a, b, c, d$ are pairwise nonisotopic, $c$ intersects only $a$
nontrivially and is disjoint from all the other curves in $P$, $d$ intersects only $b$ nontrivially and is
disjoint from all the other curves in $P$, and $c$ and $d$ are disjoint. Since $\lambda$ is injective
by Lemma \ref{inj}, $\lambda$ sends top dimensional maximal simplices to top dimensional maximal simplices, and
$\lambda([a]), \lambda([b]), \lambda([c]), \lambda([d])$ are pairwise distinct. We also have  $i(\lambda([c]), \lambda([a])) \neq 0$, $i(\lambda([c]), \lambda([x]) = 0$ for all $x \in P \setminus \{a\}$,
$i(\lambda([d]), \lambda([b])) \neq 0$, $i(\lambda([d]), \lambda([x])) = 0$ for all $x \in P \setminus \{b\}$, and
$i(\lambda([c]), \lambda([d])) = 0$ by superinjectivity. This is possible
only when $\lambda([a])$ and $\lambda([b])$ have representatives which are not adjacent w.r.t. $P'$.\end{proof}

\begin{lemma}
\label{1-sided-cn} Let $g \geq 2$. Suppose that $(g, n) = (3, 0)$ or $g+n \geq 4$.
Let $\lambda : \mathcal{C}(N) \rightarrow \mathcal{C}(N)$ be a superinjective simplicial map.
If $a$ is a 1-sided simple closed curve on $N$ whose complement is
nonorientable, then $\lambda(a)$ is the isotopy class of a 1-sided simple closed curve whose complement
is nonorientable.
\end{lemma}

\begin{proof} Let $a$ be a 1-sided simple closed curve on $N$ whose complement is nonorientable.
Let $a' \in \lambda([a])$. By Lemma \ref{easy} we know that $a'$ is a 1-sided simple closed curve. Suppose the complement of $a'$ is
orientable. This case happens only if the genus of $N$ is odd. So, suppose $g =2r +1$, where $r \geq 1$.
In this case $a$ can be put into a maximal simplex $\Delta$ of dimension $4r + n -2$.
Since $\lambda$ is injective by Lemma \ref{inj}, $\lambda(\Delta)$ is a simplex of dimension $4r + n-2$.
By using Euler characteristic arguments we see that the complement of $a'$ has genus $r$ and $n+1$
boundary components. So, in the complement of $a'$ there can be at most a $3r + n - 3$ dimensional simplex,
hence there can be at most a $3r + n - 2$ dimensional simplex containing $a'$ on $N$.  Since
dim $\lambda(\Delta) = 4r + n - 2 > 3r + n - 2$ as $r \geq 1$, we get a contradiction.
So, the complement of $a'$ is nonorientable.
\end{proof}

\section{Proof of the Main Result}

In this section we will prove our main result when $(g, n) = (3, 0)$ or $g + n \geq 5$. Together with Theorem \ref{A}
this will complete our proof of the main theorem.\\

First we consider two graphs on $N$ as given in \cite{AK}:
If $g=1$, let $\mathcal{A}$ be the set of isotopy classes of all 1-sided simple closed curves on $N$. If $g \geq 2$ let
$\mathcal{A}$ be the set of isotopy classes of all 1-sided simple closed curves which have nonorientable complements on $N$.
Let $X(N)$ be the graph with vertex set $\mathcal{A}$ such that two distinct vertices in $X(N)$ are connected by an edge if
and only if they have representatives intersecting transversely at one point.

Let $\widetilde{X}(N)$ be a subgraph of $X(N)$ with the vertex set $\mathcal{A}$. Two distinct vertices $\alpha$ and $\beta$
are connected by an edge in $\widetilde{X}(N)$ if $\alpha$ and $\beta$ have representatives a and b intersecting transversely at
one point such that

(i) either $g \geq 4$ and the surface $N_{a \cup b}$ obtained by cutting $N$ along $a$
and $b$ is connected (Since $N_a$ and $N_b$ are nonorientable, it is easy to
see that $N_{a \cup b}$ is also nonorientable in this case.),

(ii) or $1 \leq g \leq 3$ and the Euler characteristic of one of the connected
components of $N_{a \cup b}$ is at most $-2$.

We will use some connectivity results about these graphs given by Atalan-Korkmaz in \cite{AK}.\\

If $a$ is a nontrivial simple closed curve on $N$, {\it the link}, $L_a$ of $a$, is defined as the full subcomplex spanned by all
the vertices of $\mathcal{C}(N)$ which have representatives disjoint from $a$.
{\it The star} $St_a$ of $a$ is defined as the
subcomplex of $\mathcal{C}(N)$ consisting of all simplices in $\mathcal{C}(N)$ containing $[a]$ and
the faces of all such simplices. So, the link $L_a$ of $a$ is the subcomplex of
$\mathcal{C}(N)$ whose simplices are those simplices of $St_a$ which do not contain $[a]$. Let $b$ be a nontrivial
simple closed curve which is not isotopic to $a$. Then $b$ is called
{\it dual} to $a$ if they intersect transversely only once. Let $D_a$ be the set of isotopy classes of nontrivial simple closed
curves that are dual to $a$ on $N$.\\

The following theorem is given by Atalan-Korkmaz in \cite{AK}.

\begin{theorem}
\label{id} Let $N$ be a connected, nonorientable surface of genus $g \geq 1$ with $n$ holes, and $g + n \geq 5$. Let $a$
and $b$ be two 1-sided simple closed curves that are dual to each other such that $[a]$ and $[b]$ are two vertices in
$\widetilde{X}(N)$ which are connected by an edge in $\widetilde{X}(N)$. Let $h$ be a mapping class. If $h(\gamma) = \gamma$ for
every vertex $\gamma$ in the set $(St_a \cup D_a) \cap (St_b \cup D_b)$, then $h$ is the identity.
\end{theorem}

\begin{figure}
\begin{center}
\epsfxsize=2.1in \epsfbox{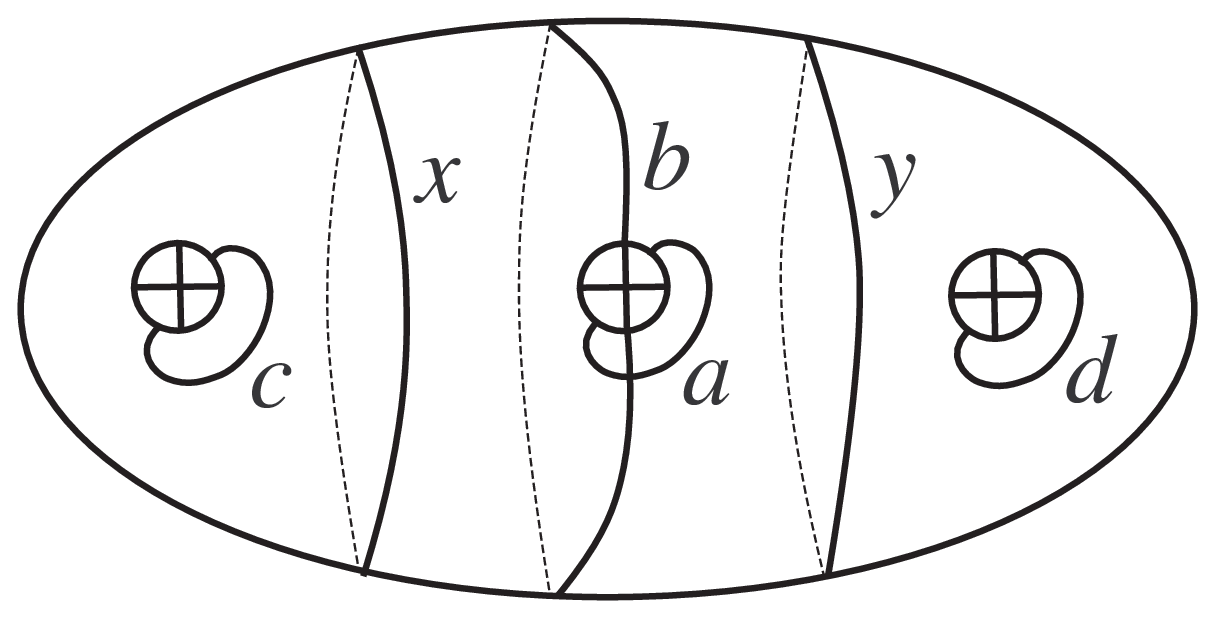} \hspace{0.15in} \epsfxsize=2.1in \epsfbox{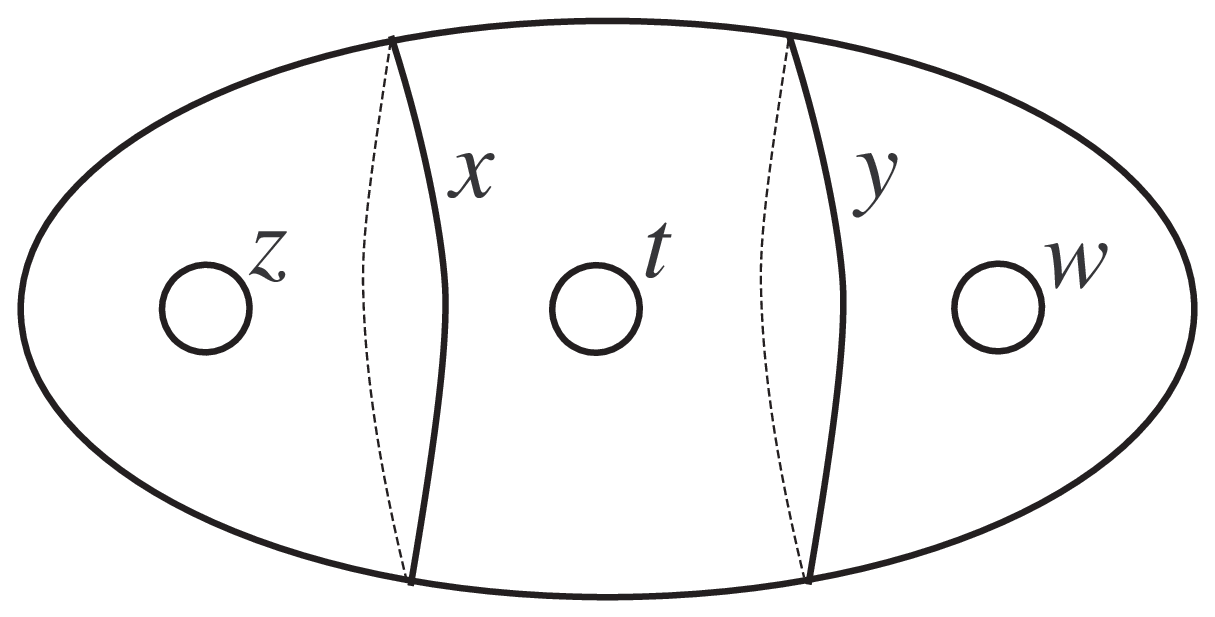}
\caption{(g, n) = (3, 0) case}
\label{Fig(3,0)}
\end{center}
\end{figure}

We prove a similar result below. Suppose $(g, n) = (3, 0)$. When we remove
the identifications on $N$ that are shown in the first part of Figure \ref{Fig(3,0)}, we get a sphere with three boundary
components, $z, t, w$, as shown in the second part. The reflection of this sphere about the $xy$-plane which changes the
orientation on each curve in $\{z, t, w, x, y\}$ shown in the figure induces a homeomorphism $R$ on $N$. We call $R$ a
reflection homeomorphism on $N$. Let $r = [R]$.

\begin{theorem}
\label{id2} Suppose $(g, n)= (3, 0)$. Let $a$, $b$ be two 1-sided simple closed curves with nonorientable complements
such that they are dual to each other. Let $h$ be a mapping class. If $h(\gamma) = \gamma$ for
every vertex $\gamma$ in the set $(St_a \cup D_a) \cap (St_b \cup D_b)$, then $h$ is the identity or $r$.
\end{theorem}

\begin{proof} Let $a$ and $b$ be two 1-sided simple closed curves with nonorientable complements such that they are dual to
each other. Let $T$ be a regular neighborhood of $a \cup b$. $T$ is a genus one surface with two boundary components,
say $x$ and $y$. Since $(g, n) = (3, 0)$, $a$ and $b$ are dual to each other, and $N_a$ and $N_b$ are nonorientable, we can see that
each of $x$ and $y$ bounds a genus one surface with one boundary component, mobius band, as shown in the Figure \ref{Fig(3,0)}.
Let $h$ be a mapping class such that $h(\gamma) = \gamma$ for every vertex $\gamma$ in the set $(St_a \cup D_a) \cap (St_b \cup D_b)$.
Let $H$ be a representative of $h$ such that $H$ fixes the curves $x, y, a$ and $b$. If $H$ fixes the orientation of $x$ then it also
fixes the orientation of $y$ by Corollary 4.6 in \cite{K2} as it also fixes $a$ and $b$ up to isotopy. Then $H$ is isotopic to
identity on $T$, and also $H$ is isotopic to identity on the two mobius bands on each side of $T$. Since the Dehn twists about $x$
and $y$ are trivial, we see that $H$ is isotopic to identity on $N$. If $H$ changes the orientation of $x$ then it also changes the
orientation of $y$ by Corollary 4.6 in \cite{K2} as it also fixes $a$ and $b$ up to isotopy. In this case the composition
$H \circ R$ fixes the orientation of both of $x$ and $y$ where $R$ is the reflection homeomorphism defined above. By the above argument
$H \circ R$ is isotopic to identity. Hence, $H$ is isotopic to $R$.\end{proof}\\

Our main results will be obtained from the following theorems.

\begin{figure}
\begin{center}
\epsfxsize=2.1in \epsfbox{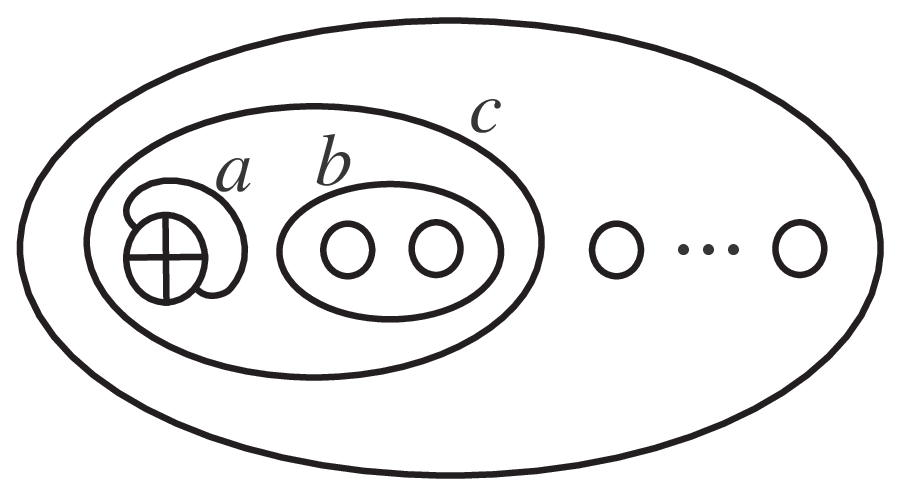} \hspace{0.15in} \epsfxsize=2.1in \epsfbox{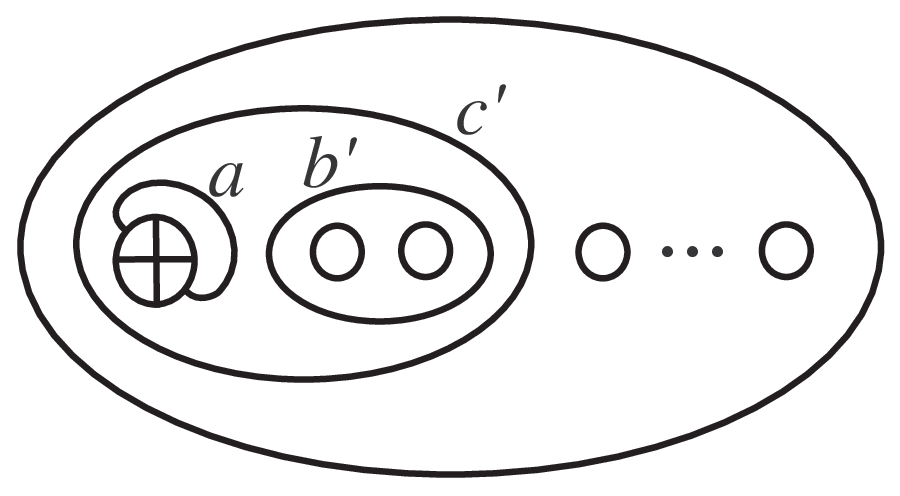}

\small{ \hspace{0.03in} (i)  \hspace{2.2in} (ii)}
\caption{Boundary correspondence}
\label{bdcor}
\end{center}
\end{figure}

\begin{theorem}
\label{New1} Let $N$ be a compact, connected, nonorientable surface of genus $g$ with
$n$ boundary components. Suppose that $g = 1$ and $n \geq 4$. If $\lambda : \mathcal{C}(N) \rightarrow \mathcal{C}(N)$ is a
superinjective simplicial map, then $\lambda$ is induced by a homeomorphism $h : N \rightarrow N$ (i.e $\lambda([a]) = [h(a)]$
for every vertex $[a]$ in $\mathcal{C}(N)$).\end{theorem}

\begin{proof} Suppose that $g=1$ and $n \geq 4$. Let $a$ be a 1-sided simple closed curve. Let $a' \in \lambda([a])$. By Lemma \ref{easy}
we know that $a'$ is a 1-sided simple closed curve. Since $g=1$, both of $N_a$ and $N_{a'}$ are orientable. So, there is a
homeomorphism $f: N \rightarrow N$ such that $f(a) = a'$. Let $f_{\#}$ be the simplicial automorphism induced by $f$ on
$\mathcal{C}(N)$. Then $f_{\#}^{-1} \circ \lambda $ fixes $[a]$. By replacing $f_{\#}^{-1} \circ \lambda $ by $\lambda$ we
can assume that $\lambda([a]) = [a]$. The simplicial map $\lambda$ restricts to $\lambda_a : L_a \rightarrow L_a$, where $L_a$
is the link of $[a]$. It is easy to see that $\lambda_a$ is a superinjective simplicial map. Since $L_a \cong \mathcal{C}(N_a)$,
we get a superinjective simplicial map $\lambda_a :  \mathcal{C}(N_a)\rightarrow \mathcal{C}(N_a)$.
By following the proof of Lemma \ref{inj}, we see that $\lambda_a$ is injective. Since $N_a$ is a sphere with at least five boundary
components, by Shackleton's result in \cite{Sh}, there is a homeomorphism
$G_a: N_a \rightarrow N_a$ such that $\lambda_a$ is induced by $(G_a)_{\#}$.

Let $\partial_a$ be the boundary component of $N_a$
which came by cutting $N$ along $a$. We can see that $G_a (\partial_a) = \partial_a$ as follows: Let $b, c$ be the curves as shown in the Figure \ref{bdcor} (i) ($b$ separates two of the boundary components of $N$). Complete $\{a, b, c\}$ to a top dimensional pair of pants decomposition $P$ on $N$. We assumed that $\lambda([a])=[a]$. Since $a, b, c$ are pairwise disjoint, there exist $b', c'$ some representatives of $\lambda([b])$ and $\lambda([c])$ respectively, such that $a, b', c'$ are pairwise disjoint. Let $P'$ be the set of pairwise disjoint representatives of $\lambda([P])$ containing $a, b', c'$. Since $a$ is adjacent to $b$ and $c$ w.r.t. $P$, $a$ is adjacent to $b'$ and $c'$ w.r.t. $P'$ by Lemma \ref{adjacent}. Since $a$ is 1-sided, there is a genus one subsurface with two boundary components, say $T$, such that $a$ is in $T$ and the boundary components of $T$ are $b', c'$. Since $b$ is adjacent to only $a$ and $c$ w.r.t. $P$, and nonadjacency is preserved by $\lambda$ by Lemma \ref{nonadjacent}, we see that $b'$ must be adjacent to only $a$ and $c'$. We know that $b'$ is a 2-sided curve as $b$ is 2-sided. This shows that $a, b', c'$ are as shown in the
Figure \ref{bdcor} (ii) (i.e. $b'$ separates two of the boundary components of $N$).
Since $\lambda_a$ is induced by $(G_a)_{\#}$, they agree on $[b]$ and $[c]$. This shows that $G_a (\partial_a) = \partial_a$.

By composing $G_a$ with a homeomorphism isotopic to identity, we can assume that $G_a$ maps antipodal points on the boundary
$\partial_a$ to antipodal points. So, $G_a$ induces a homeomorphism $g_a: N \rightarrow N$ such that $g_a (a)= a$ and
$(g_a)_{\#}$ agrees with $\lambda$ on every vertex of $[a] \cup L_a$. So, $(g_a)_{\#}^{-1} \circ \lambda$ fixes every vertex
in $[a] \cup L_a$. Let $D_a$ be the set of isotopy classes of simple closed curves that are dual to $a$.

Claim 1: $(g_a)_{\#}$ agrees with $\lambda$ on every vertex of $\{[a]\} \cup L_a \cup D_a$.

Proof of Claim 1: We already know that $(g_a)_{\#}$ agrees with $\lambda$ on every vertex of $\{[a]\} \cup L_a$. Let $d$ be a
simple closed curve that is dual to $a$. Since $g=1$, there is no 2-sided curve dual to $a$. So, $d$ has to be 1-sided.
Let $T$ be a regular neighborhood of $a \cup d$. We see that $T$ is a real projective plane
with two boundary components, say $x, y$. Since $(g_a)_{\#}^{-1} \circ \lambda$ is identity on ${[a]} \cup L_a$,
$(g_a)_{\#}^{-1} \circ \lambda$ is identity on the complement of $T$. Since $(g_a)_{\#}^{-1} \circ \lambda$
is a superinjective simplicial map, there exists $d' \in \lambda([d])$ such that $d'$ is disjoint from the
complement of $T$, so $d'$ is in $T$. Since $d$ is the only nontrivial simple closed curve in $T$ which is not
isotopic to $a$ by Scharlemann's Theorem in \cite{Sc}, we see that $d'$ is isotopic to $d$. So,
$(g_a)_{\#}^{-1} \circ \lambda ([d]) = [d]$. Hence, $(g_a)_{\#}^{-1} \circ \lambda$ is identity on ${[a]} \cup L_a \cup D_a$.

Claim 2: Let $v$ be any 1-sided simple closed curve on $N$. Then, $(g_v)_{\#} = (g_a)_{\#}$ on $\mathcal{C}(N)$.

Proof of Claim 2: Since $g=1$, $\widetilde{X}(N)$ is connected by Theorem 3.10 in \cite{AK}. So, we can find a sequence
$a \rightarrow a_1 \rightarrow a_2 \rightarrow  \cdots \rightarrow  a_n=v$ of 1-sided simple closed curves connecting
$a$ to $v$ such that each consecutive pair is connected by an edge in $\widetilde{X}(N)$. By Claim 1, $(g_a)_{\#}$ agrees with $\lambda$ on
$\{[a]\} \cup L_a \cup D_a$, and $(g_{a_1})_{\#}$ agrees with $\lambda$ on $\{[a_1]\} \cup L_{a_1} \cup D_{a_1}$. So, we see that
$(g_a)_{\#} ^{-1} (g_{a_1})_{\#}$ fixes every vertex in $([a] \cup L_a \cup D_a) \cap ([a_1] \cup L_{a_1} \cup D_{a_1})$. We note that this is the same thing as fixing every vertex in $(St_a \cup D_a) \cap (St_{a_1} \cup D_{a_1})$ since the vertex sets of them are equal. By Theorem \ref{id},
we get $(g_a)_{\#} = (g_{a_1})_{\#}$ on $\mathcal{C}(N)$. By using the sequence, we get $(g_a)_{\#} = (g_{v})_{\#}$ on $\mathcal{C}(N)$.

Since genus is 1 there is no nonseparating 2-sided curve on $N$. Since every separating curve is in the link, $L_r$, of
some 1-sided curve $r$, we see that $(g_a)_{\#}$ agrees with $\lambda$ on $\mathcal{C}(N)$.
\end{proof}

\begin{theorem}
\label{New2} Let $N$ be a compact, connected, nonorientable surface of genus $g$ with
$n$ boundary components. Suppose that $g=2$ and $n \geq 3$. If $\lambda : \mathcal{C}(N) \rightarrow \mathcal{C}(N)$ is a
superinjective simplicial map, then $\lambda$ is induced by a homeomorphism $h : N \rightarrow N$ (i.e $\lambda([a]) = [h(a)]$
for every vertex $[a]$ in $\mathcal{C}(N)$).\end{theorem}

\begin{proof} Suppose that $g=2$ and $n \geq 3$. Let $a, b$ be as in Figure \ref{bdcor2} (i). Let $P(a)$, $P(b)$ be the connected components of $\widetilde{X}(N)$ that contains $a$, $b$ respectively. By Theorem 3.10
in \cite{AK}, $\widetilde{X}(N)$ has two connected components $P(a)$ and $P(b)$. Since $\lambda$ is superinjective and $a, b$ are
pairwise disjoint, there exist
$a' \in \lambda([a])$, $b' \in \lambda([b])$ such that $a', b'$ are pairwise disjoint. By Lemma \ref{1-sided-cn} we know that
$a'$ and $b'$ are 1-sided simple closed curves with nonorientable complements since $a$ and $b$ are. So, there is a
homeomorphism $f: N \rightarrow N$ such that $f(a) = a'$. Let $f_{\#}$ be the simplicial automorphism induced by $f$ on
$\mathcal{C}(N)$. We see that $f_{\#}^{-1} \circ \lambda $ fixes $[a]$ as in the proof of Theorem \ref{New1}.
By replacing $f_{\#}^{-1} \circ \lambda $ by $\lambda$
we can assume that $\lambda([a]) = [a]$. The simplicial map $\lambda$ restricts to a superinjective map
$\lambda_{a} : L_{a} \rightarrow L_{a}$. Since $L_{a} \cong \mathcal{C}(N_{a})$, we get a superinjective simplicial map
$\lambda_{a} :  \mathcal{C}(N_{a})\rightarrow \mathcal{C}(N_{a})$. Since $N_a$ is a genus one nonorientable
surface with at least four boundary components, by Theorem \ref{New1} there is a homeomorphism $G_{a}: N_{a} \rightarrow N_{a}$
such that $\lambda_{a}$ is induced by $(G_{a})_{\#}$.

\begin{figure}
\begin{center}
\epsfxsize=2.2in \epsfbox{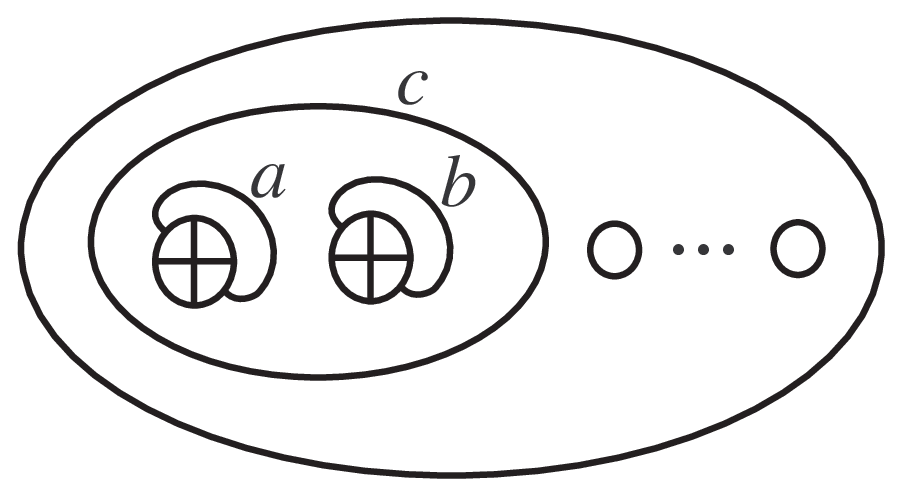} \hspace{0.17in} \epsfxsize=2.2in \epsfbox{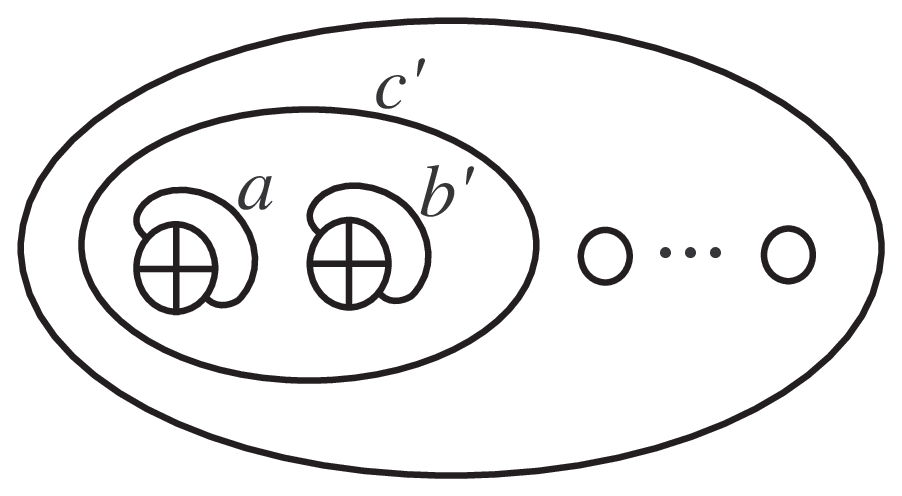}

\small{ \hspace{0.06in} (i)  \hspace{2.14in} (ii)}
\caption{Boundary correspondence}
\label{bdcor2}
\end{center}
\end{figure}

Let $\partial_a$ be the boundary component of $N_a$
which came by cutting $N$ along $a$. We can see that $G_a (\partial_a) = \partial_a$ as follows: Let $c$ be the curve as shown in the Figure \ref{bdcor2} (i). We see that $c$ separates a genus two subsurface containing $a, b$.
Complete $\{a, b, c\}$ to a top dimensional pair of pants decomposition $P$ on $N$. We assumed that $\lambda([a])=[a]$. Since $a, b, c$ are pairwise disjoint, there exist $b', c'$ some representatives of $\lambda([b])$ and $\lambda([c])$ respectively, such that $a, b', c'$ are pairwise disjoint. Let $P'$ be a set of pairwise disjoint representatives of $\lambda([P])$ containing $a, b', c'$. Since $a$ is adjacent to $b$ and $c$ w.r.t. $P$, $a$ is adjacent to $b'$ and $c'$ w.r.t. $P'$ by Lemmma \ref{adjacent}. Since $a$ is 1-sided there is a genus one subsurface with two boundary components, say $T$, such that $a$ is in $T$ and the boundary components of $T$ are $b', c'$. Since $b$ is adjacent to only $a, c$ w.r.t. $P$, and nonadjacency is preserved by $\lambda$ by Lemma \ref{nonadjacent}, we see that $b'$ must be only adjacent to and $a, c'$. We also know that $b'$ must be a 1-sided curve whose complement is nonorientable, as $b$ is such a curve (see Lemma \ref{1-sided-cn}). This shows that $a, b', c'$ are as shown in the
Figure \ref{bdcor2} (ii). Since $\lambda_a$ is induced by $(G_a)_{\#}$, they agree on $[b]$ and $[c]$. This shows that $G_a (\partial_a) = \partial_a$.

Now we continue as follows: As in the proof of Theorem \ref{New1}, by composing $G_a$ with a homeomorphism isotopic to identity, we can assume that $G_a$ maps antipodal points on the boundary
$\partial_a$ to antipodal points. So, $G_a$ induces a homeomorphism $g_a: N \rightarrow N$ such that $g_a (a)= a$ and
$(g_a)_{\#}$ agrees with $\lambda$ on every vertex of $[a] \cup L_a$. So, $(g_a)_{\#}^{-1} \circ \lambda$ fixes every vertex in $[a] \cup L_a$.

As in the proof of Theorem \ref{New1}, we get $(g_{a})_{\#}$ agrees with $\lambda$ on $\{[a]\} \cup L_{a}$ and on the isotopy class of every 1-sided curve dual to $a$. To see that it also agrees on the isotopy class of every 2-sided curve dual to $a$ we do the following:
Let $d$ be a 2-sided simple closed curve that is dual to $a$. Let $K$ be a regular neighborhood of $a \cup d$. We see that $K$ is a Klein bottle with
one boundary component, say $x$. Since $(g_a)_{\#}^{-1} \circ \lambda$ is identity on ${[a]} \cup L_a$,
$(g_a)_{\#}^{-1} \circ \lambda$ is identity on the complement of $K$. Since $(g_a)_{\#}^{-1} \circ \lambda$
is a superinjective simplicial map, there exists $d' \in \lambda([d])$ such that $d'$ is disjoint from the
complement of $K$, so $d'$ is in $K$. We know by Lemma \ref{easy} that $d'$ should be 2-sided. Since $d$ is the only 2-sided
nontrivial simple closed curve in $K$, we see that $d'$ is isotopic to $d$ (see Figure \ref{figure2}). So, $(g_a)_{\#}^{-1} \circ \lambda ([d]) = [d]$. Hence, $(g_a)_{\#}^{-1} \circ \lambda$ is identity on ${[a]} \cup L_a \cup D_a$.
So, $(g_{a})_{\#}$ agrees with $\lambda$ on $\{[a]\} \cup L_{a} \cup D_{a}$.

As in the proof of Theorem \ref{New1},
we can also see that if $v$ is any 1-sided simple closed curve such that $[v] \in P(a)$ (where $P(a)$ is the connected component of $\widetilde{X}(N)$ that contains $a$), then $(g_a)_{\#} = (g_{v})_{\#}$
on $\mathcal{C}(N)$. Hence, we can see that there exists a homeomorphism $h_1$ such that our original superinjective map $\lambda$
agrees with $(h_1)_{\#}$, on every vertex of $\{[a]\} \cup L_{a} \cup D_{a} \cup \{[v]\} \cup L_{v} \cup D_{v}$ for any 1-sided simple closed
curve $v$ such that $[v] \in P(a)$. Similarly, there exists a homeomorphism, $h_2$, such that $\lambda$ agrees with $(h_2)_{\#}$
on every vertex of $\{[b]\} \cup L_{b} \cup D_{b} \cup \{[w]\} \cup L_{w} \cup D_{w}$ for any 1-sided simple closed curve $w$ such that $[w] \in P(b)$. So, $(h_1)_{\#}^{-1} (h_2)_{\#}$ fixes everything in the intersection of these two sets, for every such $v, w$.

\begin{figure}
\begin{center}
\epsfxsize=2.5in \epsfbox{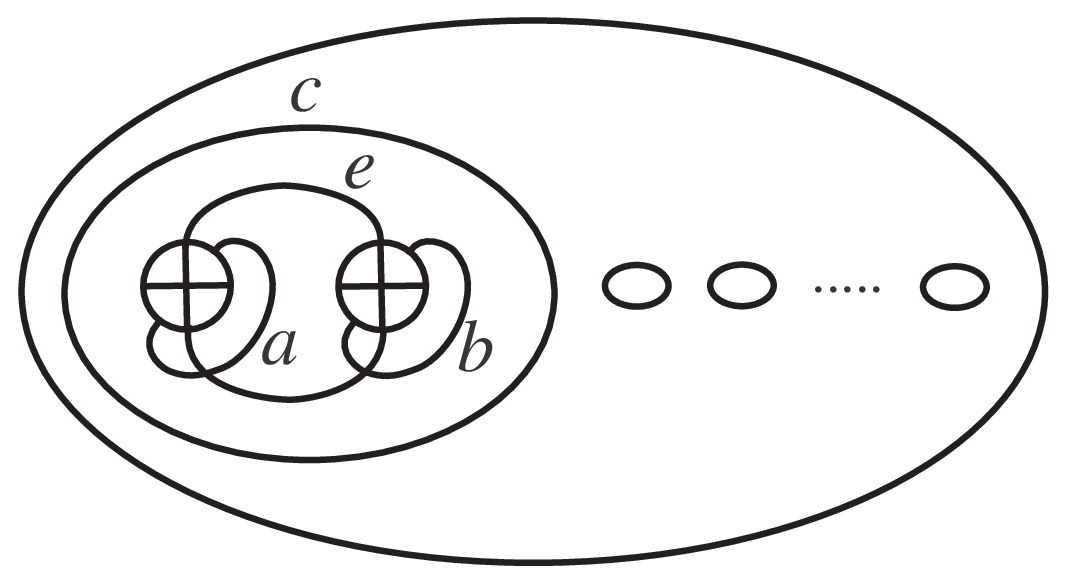} \hspace{0.1in} \epsfxsize=2.5in
\epsfbox{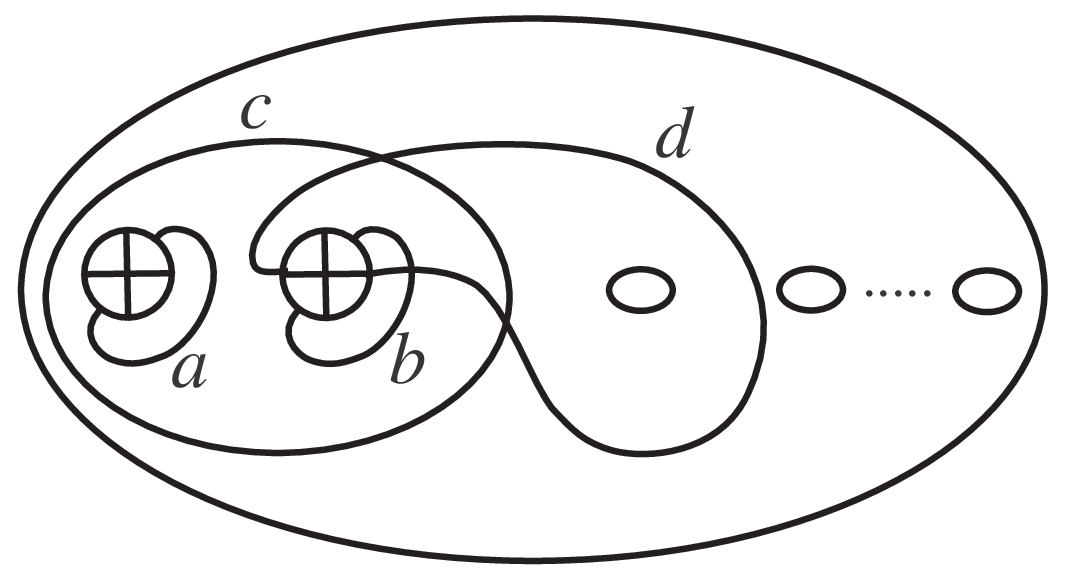}
\caption{Some curves on genus 2 surface with boundary}
\label{Fig11}
\end{center}
\end{figure}

Claim: $(h_1)_{\#} = (h_2)_{\#}$.

Proof of Claim: Let $c, d, e$ be as in Figure \ref{Fig11}. Let $T$ be the genus 2 subsurface containing $a$ and $b$ and bounded
by $c$ as shown in the figure. Since $(h_1)_{\#}^{-1} (h_2)_{\#}$ fixes $[c]$, (by composing with a map isotopic to
identity if necessary) we may assume that ${h_1}^{-1} h_2$ fixes $c$. By the above argument we know that
${h_1}^{-1} h_2$ fixes every nontrivial simple closed curve up to isotopy in the complement of $T$. Since the complement of
$T$ together with $c$ is a sphere with at least four boundary components, ${h_1}^{-1} h_2$ fixes $c$ and every nontrivial
simple closed curve up to isotopy in the complement of $T$, ${h_1}^{-1} h_2$
is isotopic to identity in the complement of $T$ by Lemma 7.1 in \cite{AK}. So, ${h_1}^{-1} h_2$ is isotopic to identity on $c$.
Since ${h_1}^{-1} h_2$ fixes $a$, $b$ up to isotopy, it fixes the curve $e$
shown in the figure up to isotopy, see the proof of Theorem \ref{A} and Figure \ref{figure2}. Recall that
the vertex set of the curve complex in $T$ is $\{[e], t_e^m ([a]) : m \in \mathbb{Z} \}$, and $e$ is the only nontrivial 2-sided curve up to isotopy
on $T$. By looking at the complex we see that
the map ${h_1}^{-1} h_2$ actually fixes every nontrivial simple closed curve on $T$ up to isotopy. Since it is also isotopic to identity on $c$,
${h_1}^{-1} h_2$ is isotopic to identity on $T$. This implies that $(h_1)_{\#}^{-1} (h_2)_{\#} = t_c^m$ for some $m \in \mathbb{Z}$. Since $(h_1)_{\#}^{-1} (h_2)_{\#}$ also fixes $[d]$, we have $m=0$, so $(h_1)_{\#}^{-1} (h_2)_{\#} = [id]$.

Let $h=h_1$. If $x$ is a 1-sided simple closed curve with nonorientable complement (i.e. $[x]$ is a vertex in $\widetilde{X}(N)$),
then $[x]$ is in one of $P(a)$ or $P(b)$ by Theorem 3.10 in \cite{AK}. So, by the above arguments $\lambda$ agrees with $h_{\#}$
on $\{[x]\} \cup L_{x} \cup D_{x}$. Since any nontrivial simple closed curve is
in the dual or link of a 1-sided simple closed curve whose complement is nonorientable, $\lambda$ agrees with
$h_{\#}$ on $\mathcal{C}(N)$.
\end{proof}

\begin{theorem}
\label{New3} Let $N$ be a compact, connected, nonorientable surface of genus $g$ with
$n$ boundary components. Suppose that $g \geq 3$, $g + n \geq 5$. If $\lambda : \mathcal{C}(N) \rightarrow \mathcal{C}(N)$ is a
superinjective simplicial map, then $\lambda$ is induced by a homeomorphism $h : N \rightarrow N$.\end{theorem}

\begin{figure}
\begin{center}
\epsfxsize=2.6in \epsfbox{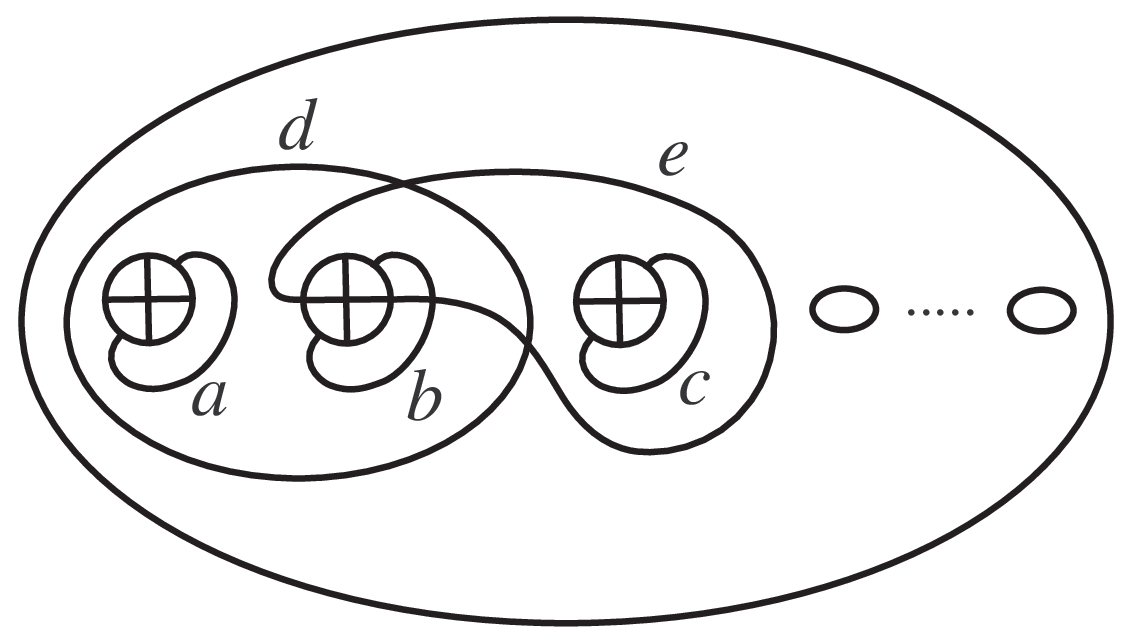} \hspace{0.1in} \epsfxsize=2.6in
\epsfbox{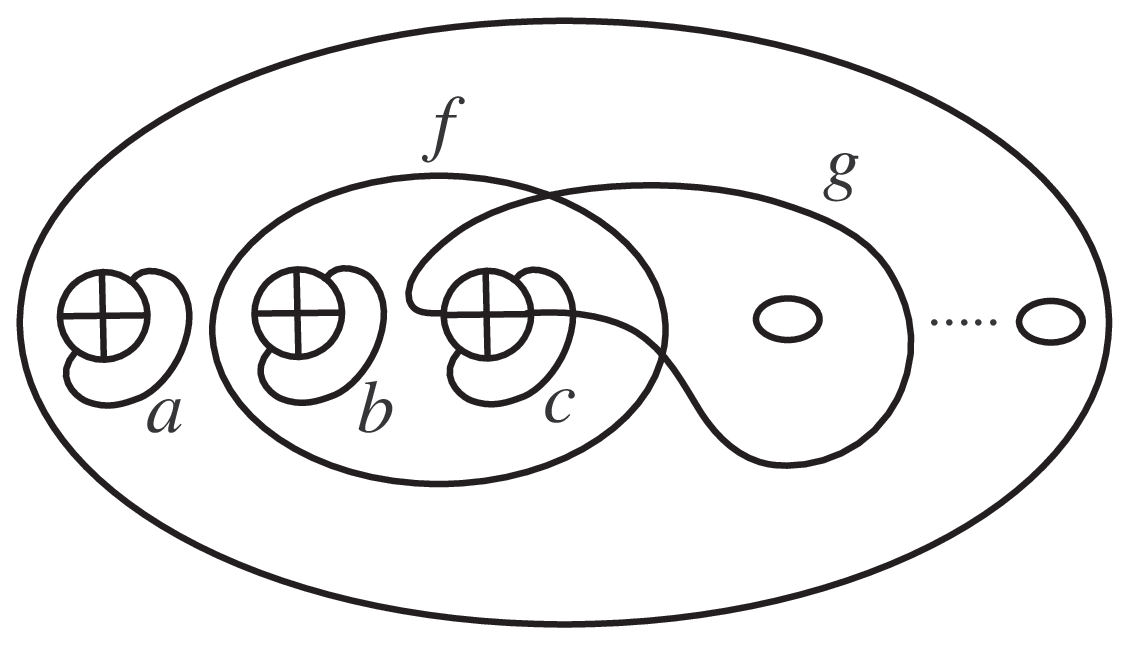}
\caption{Some curves on genus 3 surface with boundary}
\label{Fig12}
\end{center}
\end{figure}

\begin{proof} We will prove the result using induction on $g$.

Suppose that $g=3$. $N$ has at least two boundary components. Let $a, b, c$ be as in Figure \ref{Fig12}. Let $P(a)$, $P(b)$ and $P(c)$ be the connected components of $\widetilde{X}(N)$ that contains $a$, $b$ and $c$ respectively. By Theorem 3.10 in \cite{AK}, $\widetilde{X}(N)$ has three connected components $P(a)$, $P(b)$ and $P(c)$.
As in the proof of Theorem \ref{New2},
there exist homeomorphisms $h_1, h_2, h_3$ such that $\lambda$ agrees with $(h_1)_\#$ on every vertex of
$\{[a]\} \cup L_{a} \cup D_{a} \cup \{[v]\} \cup L_{v} \cup D_{v}$ for any 1-sided simple closed curve $v$ such that $[v] \in P(a)$,
$\lambda$ agrees with $(h_2)_\#$ on every vertex of $\{[b]\} \cup L_{b} \cup D_{b} \cup \{[w]\} \cup L_{w} \cup D_{w}$
for any 1-sided simple closed curve $w$ such that $[w] \in P(b)$, and $\lambda$ agrees with $(h_3)_\#$ on every vertex of
$\{[c]\} \cup L_{c} \cup D_{c} \cup \{[z]\} \cup L_{z} \cup D_{z}$ for any 1-sided simple closed curve $z$ such that $[z] \in P(c)$.
Let $d, e, f, g$ be as in Figure \ref{Fig12}. By following the proof of Theorem \ref{New2} and using the curves $d, e$ we can see that
$(h_1)_\# = (h_2)_\#$. (We note that in this case the complement of $T$ will be a genus one surface with at least 3 boundary components, and the difference map will be isotopic to identity on the complement of $T$, see Lemma 7.1 in \cite{AK}. The rest of the proof will follow as in the proof of Theorem \ref{New2}.) Similarly, by using the curves $f, g$ we can see that $(h_2)_\# = (h_3)_\#$. Hence, we have $(h_1)_\# = (h_2)_\# = (h_3)_\#$.

Let $h=h_1$. If $x$ is a 1-sided simple closed curve with nonorientable complement (i.e. $[x]$ is a vertex in $\widetilde{X}(N)$), then $[x]$ is
in one of $P(a)$ or $P(b)$ or $P(c)$ by Theorem 3.10 in \cite{AK}. By the above arguments and the proof of Theorem \ref{New2}, we see that $\lambda$ agrees with $h_{\#}$ on $\{[x]\} \cup L_{x} \cup D_{x}$. Since any nontrivial simple closed
curve is in the dual or link of a 1-sided simple closed curve whose complement is nonorientable, $\lambda$ agrees with $h_{\#}$ on $\mathcal{C}(N)$.

We will prove the remaining cases by using induction on $g$. Assume that the theorem is true for some genus $g-1$ where $g-1 \geq 3$. We will prove that it is true for genus $g$. Let $a$ be a 1-sided simple closed curve on $N$ such that $N_a$ is nonorientable.
Let $a' \in \lambda([a])$. By Lemma \ref{1-sided-cn}, $a'$ is a 1-sided simple closed curve with nonorientable complement.
There is a homeomorphism $f: N \rightarrow N$ such that $f(a) = a'$. Let $f_{\#}$ be the simplicial automorphism induced by $f$ on $\mathcal{C}(N)$. Then $f_{\#}^{-1} \circ \lambda$ fixes $[a]$. By replacing $f_{\#}^{-1} \circ \lambda $ by $\lambda$ we can assume
that $\lambda([a]) = [a]$. The simplicial map $\lambda$ restricts to a superinjective simplicial map $\lambda_a : L_a \rightarrow L_a$
as before. Since $L_a \cong \mathcal{C}(N_a)$, we get a superinjective simplicial map $\lambda_a :  \mathcal{C}(N_a)\rightarrow \mathcal{C}(N_a)$.
By the induction assumption, there is a homeomorphism $G_a: N_a \rightarrow N_a$ such that $\lambda_a$ is induced by $(G_a)_{\#}$.

By following the proof of Theorem \ref{New1}, we see the following: There is a homeomorphism $g_a$ such that $\lambda$ agrees with $(g_a)_\#$ on $\{[a]\} \cup L_a \cup D_a$, and if $v$ is a 1-sided simple closed curve with nonorientable complement such that $[v]$ is connected to $[a]$ by a path in $\widetilde{X}(N)$, then $(g_v)_{\#} = (g_a)_{\#}$ on $\mathcal{C}(N)$. We also see that $\lambda$ agrees with $(g_v)_\#$ on $\{[v]\} \cup L_v \cup D_v$. Let $w$ be any 1-sided simple closed curve with nonorientable complement. Between $[a]$ and $[w]$ there is a path in $\widetilde{X}(N)$ since $\widetilde{X}(N)$ is connected by Theorem 3.10 in \cite{AK}. So, we have $(g_a)_{\#} = (g_w)_{\#}$ on $\mathcal{C}(N)$. Since the isotopy class of every nontrivial simple closed curve is in the link or dual of some 1-sided simple closed curve with nonorientable complement,
we see that $\lambda$ agrees with $(g_a)_\#$ on $\mathcal{C}(N)$.
By induction, we get our result.\end{proof}

\begin{figure}
\begin{center}
\epsfxsize=2.5in \epsfbox{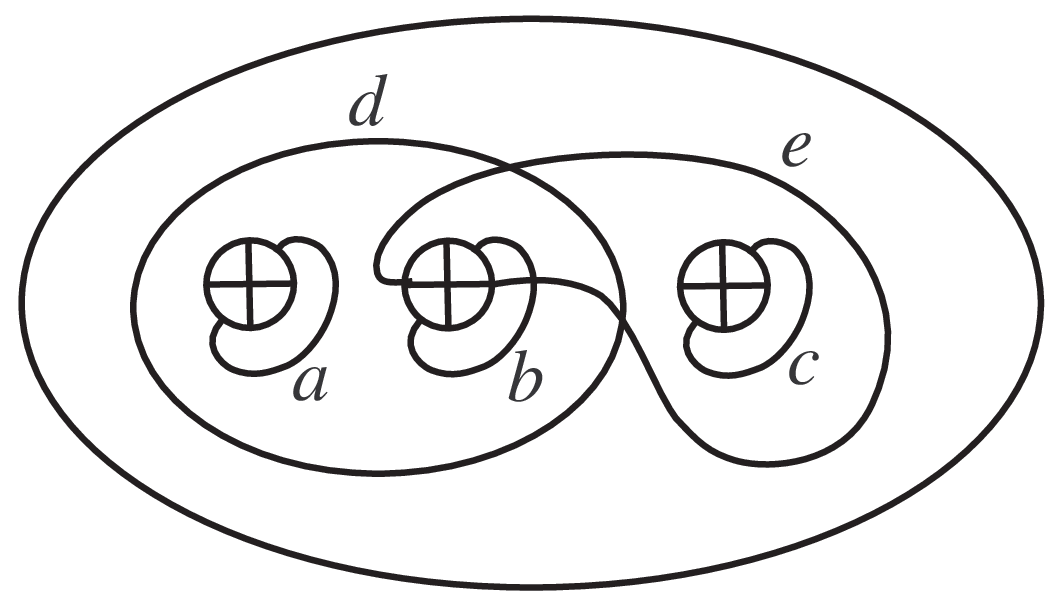}
\caption{Some curves on genus 3 surface}
\label{Figlast}
\end{center}
\end{figure}

\begin{theorem}
\label{New4} Let $N$ be a closed, connected, nonorientable surface of genus 3. If $\lambda : \mathcal{C}(N) \rightarrow \mathcal{C}(N)$ is a
superinjective simplicial map, then $\lambda$ is induced by a homeomorphism $h : N \rightarrow N$ (i.e $\lambda([a]) = [h(a)]$
for every vertex $[a]$ in $\mathcal{C}(N)$).\end{theorem}

\begin{proof} Suppose $(g, n) = (3, 0)$. Consider the curves $a, b, c$ as shown in Figure \ref{Figlast}.
Let $Q(a), Q(b), Q(c)$ be the connected components of $a, b, c$ in $X(N)$ respectively. By using Lemma \ref{1-sided-cn}, Theorem \ref{A},
Theorem \ref{id2}, and following the proof of Theorem \ref{New3}, we see that there exist homeomorphisms $h_1, h_2, h_3$ such that $\lambda$
agrees with $(h_1)_\#$ on $\{[a]\} \cup L_{a} \cup D_{a} \cup \{[v]\} \cup L_{v} \cup D_{v}$ for any 1-sided simple closed curve
$v$ such that $[v] \in Q(a)$, $\lambda$ agrees with $(h_2)_\#$ on $\{[b]\} \cup L_{b} \cup D_{b} \cup \{[w]\} \cup L_{w} \cup D_{w}$
for any 1-sided simple closed curve $w$ such that $[w] \in Q(b)$, and $\lambda$ agrees with $(h_3)_\#$ on
$\{[c]\} \cup L_{c} \cup D_{c} \cup \{[z]\} \cup L_{z} \cup D_{z}$ for any 1-sided simple closed curve $z$ such that $[z] \in Q(c)$.
Let $d, e$ be as in Figure \ref{Figlast}. By following the proof of Theorem \ref{New3}, and using the curves $d, e$ we can see that
$(h_1)_\# = (h_2)_\#$ or $(h_1)_\# = (h_2)_\# \circ R_\#$ where $R$ is the reflection homeomorphism that we described at the beginning of
this section. Similarly, we can see that $(h_2)_\# = (h_3)_\#$ or $(h_2)_\# = (h_3)_\# \circ R_\#$. Since $R_\#$
fixes the isotopy class of every nontrivial simple closed curve, the action of all these maps $(h_1)_\#$, $(h_2)_\#$ and $(h_3)_\#$
are the same on $\mathcal{C}(N)$.

Let $h=h_1$. If $x$ is a 1-sided simple closed curve with nonorientable complement, then $[x]$ is in one of $Q(a)$ or $Q(b)$ or $Q(c)$ by
Theorem 3.9 in \cite{AK}. By the above arguments and the proof of Theorem \ref{New1}, we see that $\lambda$
agrees with $h_{\#}$ on $\{[x]\} \cup L_{x} \cup D_{x}$. Since any nontrivial simple closed curve is in the dual or link of
a 1-sided simple closed curve whose complement is nonorientable, $\lambda$ agrees with $h_{\#}$ on $\mathcal{C}(N)$. This completes the
proof.\end{proof}\\

Combining our results in Theorem \ref{A}, Theorem \ref{New1}, Theorem \ref{New2}, Theorem \ref{New3} and Theorem \ref{New4} we get our main result:

\begin{theorem} Let $N$ be a compact, connected, nonorientable surface of genus $g$ with
$n$ boundary components. Suppose that either $(g, n) \in \{(1, 0), (1, 1), (2, 0), (2, 1), (3, 0)\}$ or $g + n \geq 5$.
If $\lambda : \mathcal{C}(N) \rightarrow \mathcal{C}(N)$ is a superinjective simplicial map, then $\lambda$ is
induced by a homeomorphism $h : N \rightarrow N$.\end{theorem}

\vspace{0.2cm}

{\bf Acknowledgments}
\vspace{0.3cm}

We thank Peter Scott, Nikolai Ivanov, Mustafa Korkmaz and Ferihe Atalan for some discussions and their comments about this paper.
We also thank the referee for the comments about the paper.

eirmak@bgsu.edu

Bowling Green State University

Department of Mathematics and Statistics

Bowling Green, 43403, OH
\end{document}